\documentclass[reqno]{amsart}

\usepackage{amsmath,amsthm,amsfonts,amssymb}
\usepackage{mathtools}
\usepackage{graphicx}
\usepackage{epstopdf}
\usepackage{enumerate}
\allowdisplaybreaks

\newtheorem{theorem}{Theorem}
\newtheorem{lemma}{Lemma}
\newtheorem{corollary}{Corollary}

\theoremstyle{definition}
\newtheorem{definition}{Definition}

\theoremstyle{remark}
\newtheorem{remark}{Remark}

\newcommand{\C}{\mathbb{C}}
\newcommand{\R}{\mathbb{R}}
\newcommand{\Z}{\mathbb{Z}}
\newcommand{\pl}{\mathbb{C}P^1}
\newcommand{\Q}{\mathbb{Q}}
\newcommand{\Qb}{\overline{\mathbb{Q}}}
\newcommand{\Dc}{\mathcal{D}}

\newcommand{\G}{\Gamma}

\let\a\alpha
\let\b\beta
\let\d\delta
\let\s\sigma
\let\D\partial
\let\O\Omega

\def\odd{{\rm odd}}
\def\res{{\rm Res}}
\def\pd#1#2{\frac{\D#1}{\D#2}}
\def\tF{{\widetilde F}}
\def\tG{{\widetilde G}}
\def\tL{{\widetilde L}}
\def\tT{{\widetilde T}}
\def\tf{{\tilde f}}

\begin{document}
\date{}

\author{M.~Kazarian,~P.~Zograf}
\address{Steklov Mathematical Institute\\8 Gubkin St.\\Moscow 119991 Russia}
\email{kazarian{\char'100}mccme.ru}
\address{St.Petersburg Department of the Steklov Mathematical Institute\\
Fontanka 27\\ St. Petersburg 191023, and
Chebyshev Laboratory of St. Petersburg State University\\
14th Line V.O. 29B\\
St.Petersburg 199178 Russia}
\email{zograf{\char'100}pdmi.ras.ru}

\title[Virasoro constraints and topological recursion]{Virasoro constraints and topological recursion for Grothendieck's dessin counting}


\begin{abstract}
We compute the number of coverings of ${\C}P^1\setminus\{0, 1, \infty\}$ with a given monodromy type over~$\infty$ and given numbers of preimages of 0 and 1. We show that the generating function for these numbers enjoys several remarkable integrability properties: it obeys the Virasoro constraints, an evolution equation, the KP (Kadomtsev-Petviashvili) hierarchy and satisfies a topological recursion in the sense of Eynard-Orantin.
\end{abstract}

\maketitle

\tableofcontents

\section{Introduction and preliminaries}
Enumerative problems arising in various fields of mathematics, from combinatorics and representation theory to algebraic geometry and low-dimensional topology, often bear much in common. In many cases the generating functions associated with these problems exhibit similar behavior -- in particular, they may satisfy 
\begin{itemize}
\item
Virasoro constraints,
\item
Evolution equations of the ``cut-and-join'' type,
\item
Integrable hierarchy (such as Kadomtsev-Petviashvili (KP), Korteveg-\linebreak DeVries (KdV) or Toda equations),
\item
Topological recursion (also known as Eynard-Orantin recursion).
\end{itemize}
Simple Hurwitz numbers provide one of the best studied examples of such an enumerative problem -- indeed, their generating function satisfies the celebrated cut-and-join equation \cite{GJ1}, the Virasoro constraints (via the ELSV theorem \cite{ELSV} 
and the famous Mumford's Grothendieck-Riemann-Roch formula \cite{M} it reduces to the Witten-Kontsevich potential), the KP hierarchy \cite{O}, \cite{KL} or \cite{K}, and the topological recursion \cite{EMS}. Other examples include the Witten-Kontsevich theory, Mirzakhani's Weil-Petersson volumes, Gromov-Witten invariants of the complex projective line, invariants of knots, etc. (see \cite{EO1}, \cite{EO2} for a review).

These remarkable integrability properties of generating functions usually result from matrix model reformulations of the corresponding counting problems. However, in this paper we show that for the enumeration of Grothendieck's {\em dessins d'enfants} all these properties follow from pure combinatorics in a rather straightforward way. 

The origin of Grothendieck's theory of dessins d'enfants~\cite{G} lies in the famous result by Belyi:

\begin{theorem} {\rm (Belyi,~\cite{B})}
A smooth complex algebraic curve~$C$ is defined over the field of algebraic numbers~$\Qb$ if and only if there exist a non-constant meromorphic function~$f$ on~$C$ (or a holomorphic branched cover~$f:C\to\pl$) that is ramified only over the points $0,1,\infty\in\pl$.
\end{theorem}

We call~$(C,f)$, where~$C$ is a smooth complex algebraic curve and~$f$ is a meromorphic function on~$C$ unramified over~$\pl\setminus\{0,1,\infty\}$, a {\em Belyi pair}. For a Belyi pair~$(C,f)$ denote by~$g$ the genus of~$C$ and by~$d$ the degree of~$f$. Consider the inverse image~$f^{-1}([0,1])\subset C$ of the real line segment $[0,1]\subset\pl$. This is a connected bicolored graph with~$d$ edges, whose vertices of two colors are the preimages of 0 and 1 respectively, and the ribbon graph structure is induced by the embedding $f^{-1}([0,1])\hookrightarrow C$. (Recall that a ribbon graph structure is given by prescribing a cyclic order of half-edges at each vertex of the graph.)
The following is straightforward (cf. also~\cite{LZ}):

\begin{lemma}\label{Gr}{\rm (Grothendieck,~\cite{G})}
There is a one-to-one correspondence between the isomorphism classes of Belyi pairs and connected bicolored ribbon graphs.
\end{lemma}

\begin{definition}
A connected bicolored ribbon graph representing a Belyi pair is called Grothendieck's {\em dessin d'enfant}.\footnote{An important observation of Grothendieck that, by Belyi's theorem, the absolute Galois group ${\rm Gal}(\Qb/\Q)$ naturally acts on dessins,
lies beyond the scope of this paper; we refer the reader to \cite{LZ} for details.}
\end{definition}

Let~$(C,f)$ be a Belyi pair of genus~$g$ and degree~$d$, and let~$\G=f^{-1}([0,1])\hookrightarrow C$ be the corresponding dessin. Put $k=|f^{-1}(0)|,\;l=|f^{-1}(1)|$ and $m=|f^{-1}(\infty)|$, then we have $2g-2=d-(k+l+m)$. We assume that the poles of $f$ are labeled and denote the set of their orders by $\mu=(\mu_1,\ldots,\mu_m)$, so that $d=\sum_{i\geq 1}\mu_i$. The triple~$(k,l,\mu)$ will be called here the {\em type} of the dessin~$\G$, and the set of all dessins of type~$(k,l,\mu)$ will be denoted by~$\Dc_{k,l;\mu}$.

Actually, instead of the dessin $\G=f^{-1}([0,1])$ corresponding to a Belyi pair~$(C,f)$ it is more convenient to consider the graph $\G^*=\overline{f^{-1}(1/2+\sqrt{-1}\R)}$ dual to $\Gamma$ (where the bar denotes the closure in $C$), see Fig.~\ref{dual}. The graph $\G^*$ is connected, has $m$ ordered vertices of even degrees $2\mu_1,\ldots,2\mu_m$ at the poles of $f$ and inherits a natural ribbon graph structure. Moreover, the boundary components (faces) of $\G^*$ are naturally colored: a face is colored in white (resp. in gray) if it contains a preimage of 0 (resp. 1), and every edge of $\G^*$ belongs to precisely two boundary components of different color.

\begin{figure}[hbt]%
 \begin{center}
 \includegraphics[width=5cm]{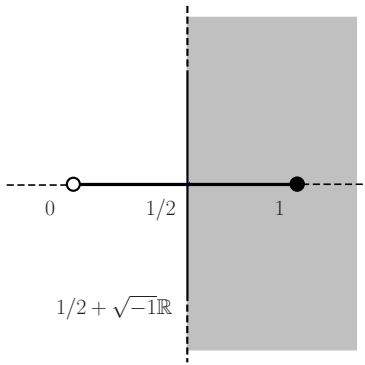}
 \caption{Decomposition of $\pl$ into two 1-gons.}
 \label{dual}%
\end{center}
\end{figure}

In this paper we are interested in the weighted count of labeled dessins d'enfants of a given type. Namely, define
\begin{align*}
N_{k,l}(\mu)=N_{k,l}(\mu_1,\ldots,\mu_m)=\sum_{\G\in\Dc_{k,l,\mu}}\frac{1}{|{\rm Aut}_b \G|}\;,
\end{align*}
where ${\rm Aut}_b \G$ denotes the group of automorphisms of $\G$ that preserve the boundary componentwise.\footnote{Equivalently, we
can put $N_{k,l}(\mu)=\sum_{\G\in\Dc_{k,l,\mu}}\frac{1}{|{\rm Aut}_v \G^*|}\;,$
where ${\rm Aut}_v \G^*$ is the group of automorphisms of the dual graph $\G^*$ preserving each vertex pointwise. A closely related problem of the weighted count of unlabeled dessins $\G$ with weights $\frac{1}{|{\rm Aut}\,\G|}$ is equivalent to the above one.
If one treats $\mu$ as the unordered partition $[1^{m_1}2^{m_2}\ldots]$, where $m_j=\#\{\mu_i=j\}$, then the corresponding number of dessins of type $(k,l,\mu)$ is equal to $\frac{1}{|{\rm Aut}\,\mu|}\,N_{k,l}(\mu)$ with
$|{\rm Aut}\,\mu|=m_1!m_2!\ldots\;.$}
Consider the total generating function
\begin{align}\label{gf}
F(s,u,v,p_1,p_2,\dots) = \sum_{k,l,m\geq 1}\frac{1}{m!}\sum_{\mu\in\Z_+^m} N_{k,l}(\mu) s^{d} u^k v^l\, p_{\mu_1}\ldots p_{\mu_m}\;,
\end{align}
where the second sum is taken over all ordered sets $\mu=(\mu_1,\ldots,\mu_m)$ of positive integers, and
$d=\sum_{i=1}^m \mu_i$.

The objective of this paper is to show that the generating function $F$ satisfies all four integrability properties listed at the beginning of this section -- namely, Virasoro constraints, an evolution equation, the KP (Kadomtsev-Petviashvili) hierarchy and a topological recursion. We prove the Virasoro constraints by a bijective combinatorial argument and derive from them all other properties of $F$.\footnote{While this paper was in preparation, similar results were independently obtained by matrix integration methods in \cite{AC} and generalized further in \cite{AMMN}.}  As a result, we obtain a simpler version of the topological recursion in terms of homogeneous components of $F$. We also revisit the problem of enumeration of the ribbon graphs with a prescribed boundary type. Topological recursion for this problem was first established in \cite{EO2} (cf. also \cite{DMSS}). In this paper we give a different, more streamlined proof of it based on the Virasoro constraints and show that the corresponding generating function satisfies an evolution equation and the KP hierarchy as well. These (and other) examples convincingly demonstrate that Virasoro constraints imply topological recursion and are in fact equivalent to it.

Additionally, we show how our results can be applied to effectively enumerate orientable maps and hypermaps regardless of the boundary type. In particular, we present a very straightforward derivation of the famous Harer-Zagier recursion \cite{HZ} for the numbers of genus $g$ polygon gluings from the Walsh-Lehman formula \cite{WL} (a higher genus generalization of Tutte's recursion \cite{T}).\footnote{Recently we came across the paper \cite{CC} where this result was proven along similar lines. However, the authors of \cite{CC} do not explicitly use Virasoro constraints that considerably simplify and clarify the proof.}

\section{Virasoro constraints}

\subsection{Virasoro constraints for the numbers of dessins}

For any integer~$n\geq 0$ consider the differential operator
\begin{multline}
L_n=-\frac{n+1}{s}\frac{\partial}{\partial p_{n+1}}+(u+v)n\frac{\partial}{\partial p_{n}}
+\sum_{j=1}^\infty p_j(n+j)\frac{\partial}{\partial p_{n+j}}\\
{}+\sum_{i+j=n}ij\frac{\partial^2}{\partial p_{i} \partial p_{j}}+\delta_{0,n}uv\;.\label{V}
\end{multline}

A straightforward check shows that for any integer~$m,n\geq 0$
\begin{align*}
[L_m, L_n]=(m-n)L_{m+n}\,.
\end{align*}
In other words, the operators~$L_n$ form (a half of) a representation of the Virasoro (or, rather, Witt) algebra.

The main technical statement of this section is the following

\begin{theorem}\label{Virasoro}
The partition function $e^F=e^{F(s,u,v,p_1,p_2,\dots)}$ satisfies the infinite system of non-linear differential equations (Virasoro constraints)
\begin{align}\label{cons}
L_ne^F=0\;.
\end{align}
The equations~\eqref{cons} determine the partition function~$e^F$ uniquely.
\end{theorem}

\begin{proof}
The Virasoro constraints~\eqref{cons} can be re-written as follows:
\begin{align}
\frac{n+1}{s}\frac{\partial F}{\partial p_{n+1}}&=
\sum_{j=1}^\infty p_j(n+j)\frac{\partial F}{\partial p_{n+j}}+(u+v)n\frac{\partial F}{\partial p_{n}}\nonumber\\
&+\sum_{i+j=n}ij\left(\frac{\partial^2 F}{\partial p_{i} \partial p_{j}}+
\frac{\partial F}{\partial p_{i}} \frac{\partial F}{\partial p_{j}}\right) + \delta_{0,n}uv\;.\label{vc}
\end{align}
Eq.~(\ref{vc}) for $n+1=\mu_1$ can be further re-written as a recursion relation for the coefficients $N_{k,l}(\mu)$ of $F$:
\begin{align}
\mu_1\,N_{k,l}(\mu_1,&\ldots,\mu_m)
=\sum_{j=2}^m (\mu_1+\mu_j-1)N_{k,l}(\mu_1+\mu_j-1,\mu_2,\ldots,\widehat{\mu}_j,\ldots,\mu_m)\nonumber\\
&{}+(\mu_1-1)(N_{k-1,l}(\mu_1-1,\mu_2,\ldots,\mu_m)+N_{k,l-1}(\mu_1-1,\mu_2,\ldots,\mu_m))\nonumber\\
&{}+\sum_{i+j=\mu_1-1}ij \bigg(N_{k,l}(i,j,\mu_2,\ldots,\mu_m)\nonumber\\
&\qquad{}+\mathop{\sum_{k_1+k_2=k}}_{l_1+l_2=l}\quad\sum_{I\sqcup J=\{2,\ldots,m\}}
N_{k_1,l_1}(i,\mu_I)N_{k_2,l_2}(j,\mu_J)\bigg)\;,\label{vt}
\end{align}
where $\mu_I=\mu_{i_1},\ldots,\mu_{i_k},\;I=\{i_1,\ldots,i_k\}$, and the hat means that the corresponding term is omitted.\footnote{Formula (\ref{vt}) is a ``bicolored" analogue of Tutte's recursion, cf. \cite{T}, Eq.~2.1, for $g=0$ and \cite{WL}, Eq.~(6), for any $g\geq 0$ (a more general form of Tutte's recursion one can find, e.~g., in \cite{EO2}).} This recursion is valid for $\sum_{i=1}^m\mu_i>1$ and expresses the numbers~$N_{k,l}(\mu)$ recursively in terms of~$N_{1,1}(1)=1$.

We prove this recursion similar to \cite{WL} (cf. also \cite{DMSS}, \cite{EO2}, \cite{N}) by establishing a direct bijection between dessins counted in the left and right hand sides of~\eqref{vt}. Here it is more convenient to deal with the dual graphs instead.
Let~$\G^*$ be the ribbon graph dual to a dessin $\G$ of type $(k,l,\mu)$. There are $2\mu_1$ ways to pick a half-edge incident to the first vertex of $\G^*$. Following \cite{DMSS} we label this half-edge with an arrow (labeling of half-edges allow us to forget about nontrivial automorphisms). When $\G$ varies over the set $\Dc_{k,l,\mu}$, this gives twice the number in the l.h.s. of~\eqref{vt}.

Let us now express the same number in terms of dessins with one edge less. This can be done by contracting (or expanding) the labeled edges in the dual graphs in a way that preserves the proper coloring of faces. The following possibilities can occur:

\begin{enumerate}[(i)]
\item The labeled edge connects the first vertex with the $j$-th vertex, $j\neq 1$. Contracting this edge we get a ribbon graph with properly bicolored faces of type $(k,l,\mu_1+\mu_j-1,\mu_2,\ldots,\widehat{\mu}_j,\ldots,\mu_m)$, see~Fig.~\ref{contract}.
Conversely, given a graph of type $(k,l,\mu_1+\mu_j-1,\mu_2,\ldots,\widehat{\mu}_j,\ldots,\mu_m)$, there are $2(\mu_1+\mu_j-1)$ ways to split its first vertex into two ones of degrees $2\mu_1$ and $2\mu_j$. Since $j$ can vary from 2 to $m$, this gives twice the first sum in the r.h.s. of ~\eqref{vt}.

\begin{figure}[hbt]%
 \begin{center}
 \includegraphics[width=9cm]{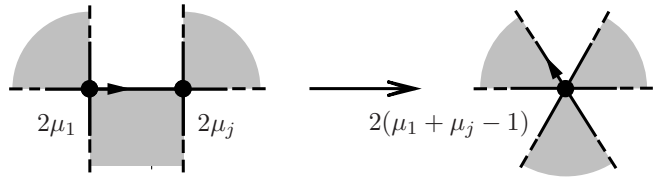}
 \caption{Contracting an edge with different endpoints.}
 \label{contract}%
\end{center}
\end{figure}

\item The labeled edge forms a loop that bounds a white 1-gon, see~Fig.~\ref{loop}. Contracting such a loop we reduce both $k$ and $\mu_1$ by 1, leaving $l$ and $\mu_j,\; j=2,\ldots,m,$ unchanged. Conversely, if we have a graph of type $(k-1,l,\mu_1-1,\mu_2,\ldots,\mu_m$, we can insert a loop into any of the $\mu_1-1$ gray sectors at the first vertex in order to get a graph of type $(k,l,\mu_1,\ldots,\mu_m)$, and 2 ways to label one of its half-edges. The case of a loop bounding a gray 1-gon can be treated verbatim, giving twice the second term in the r.h.s. of~\eqref{vt}.

\begin{figure}[hbt]%
 \begin{center}
 \includegraphics[width=8cm]{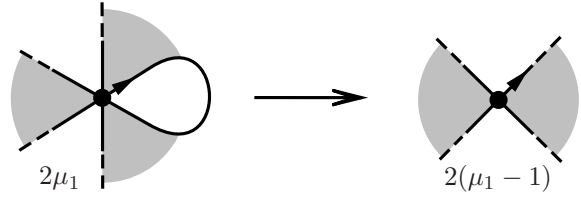}
 \caption{Contracting a loop that bounds a 1-gon.}
 \label{loop}%
\end{center}
\end{figure}

\item The labeled edge forms a loop whose half-edges are not adjacent relative to the cyclic order of half-edges at the first vertex. Contracting such a loop we split the first vertex into two ones, say, of degrees $2i$ and $2j$, where $i+j=\mu_1-1$, see Fig.~\ref{split}. Under this operation the graph may remain connected, or may split into two connected components. In the former case we get a graph of type $(k,l,i,j,\mu_2,\ldots,\mu_m)$. Reversing this operation, we join the first two vertices and add a loop. We can place the labeled half-edge of the loop in any of the $2i$ sectors at the first vertex, but its other half-edge can be placed only in one of $j$ sectors of different color at the second vertex (otherwise it will not be compatible with the face coloring). This gives us twice the third term in the r.h.s. of~\eqref{vt}. The latter case when the graph becomes disconnected can be treated similarly. 

\begin{figure}[hbt]%
 \begin{center}
 \includegraphics[width=8cm]{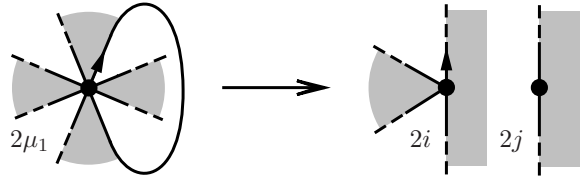}
 \caption{Contracting a loop that splits the vertex into two ones of degrees $2i$ and $2j$ with $i+j=\mu_1-1$.}
 \label{split}%
\end{center}
\end{figure}
\end{enumerate}

The operations (i)--(iii) are reversible and compatible with the face coloring, thus establishing a required bijection. This proves the Virasoro constraints~\eqref{vc}. It is also not hard to see that the Virasoro constraints determine the partition function~$e^F$ uniquely, since they are equivalent to the recursion~\eqref{vt}. 
\end{proof}

\begin{corollary}\label{evo}
Put
\begin{align*}
\Lambda_1&=\sum_{i=2}^\infty (i-1)p_i\,\frac{\partial}{\partial p_{i-1}}\;,\nonumber\\
M_1&=\sum_{i=2}^\infty \sum_{j=1}^{i-1} \left((i-1)p_j p_{i-j}\,\frac{\partial}{\partial p_{i-1}}
+ j(i-j) p_{i+1}\,\frac{\partial^2}{\partial p_j \partial p_{i-j}}\right)\;.
\end{align*}
Then the partition function~$e^F$ satisfies the evolution equation
\begin{align}
\frac{\partial e^F}{\partial s}=((u+v)\Lambda_1+M_1+uvp_1)e^F\;,\label{eveq}
\end{align}
and is uniquely determined by the initial condition $F\left|_{s=0}\right.=0$. In other words,~$e^F$ is explicitly given by the formula
\begin{align*}
e^F=e^{s((u+v)\Lambda_1+M_1+uvp_1)}\,1
\end{align*}
were ``1'' stands for the constant function identically equal to~1.
\end{corollary}
\begin{proof}
Multiply the both sides of~\eqref{V} by~$p_{n+1}$ and sum over~$n$. We get
\begin{align*}
\sum_{n=0}^\infty p_{n+1}L_n
&=\sum_{n=0}^\infty p_{n+1}\left(-\frac{n+1}{s}\frac{\partial}{\partial p_{n+1}}+(u+v)n\frac{\partial}{\partial p_{n}}\right.\nonumber\\
&\hspace{0.4in}+\sum_{j=1}^\infty p_j(n+j)\frac{\partial}{\partial p_{n+j}}
+\sum_{i+j=n}\left. ij\frac{\partial^2}{\partial p_{i} \partial p_{j}}\right)+uvp_1\nonumber\\
&=-\frac{\partial}{\partial s}+(u+v)\Lambda_1+M_1+uvp_1\;,
\end{align*}
and the required statement immediately follows from~\eqref{cons}.
\end{proof}

\begin{remark}
A different proof of Corollary \ref{evo} has recently appeared in \cite{Z}.
\end{remark}

\begin{remark}\label{ps}
Denote by $\psi$ the principal specialization of the partition function $e^F$:
\begin{align*}
\psi=\psi(s,t,u,v)=e^{F(s,u,v,p_1,p_2,\ldots)}\left|_{p_i=t^i}\right.\;,
\end{align*}
where $t$ is a new formal variable. It is not hard to check that
\begin{align*}
\Lambda_1 e^F\left|_{p_i=t^i}\right.=t\left(t\frac{d}{dt}\right)\psi\,, \qquad M_1 e^F\left|_{p_i=t^i}\right.=t\left(t\frac{d}{dt}\right)^2\psi\,.
\end{align*}
Then, with the help of the obvious identity $s\frac{\D\psi}{\D s}=t\frac{\D\psi}{\D t}$, the evolution equation \eqref{eveq} translates into the following equation for the wave function $\psi$:
\begin{align*}
\frac{1}{st}\left(t\frac{d}{dt}\right)\psi=(u+v)\left(t\frac{d}{dt}\right)\psi+\left(t\frac{d}{dt}\right)^2\psi+uv\psi\;.
\end{align*}
It can be further rewritten as the Schr\"odinger equation
\begin{align}\label{qc}
t^2\,\frac{d^2\psi}{dt^2}+\left((u+v+1)\,t-\frac{1}{s}\right)\,\frac{d\psi}{dt}+uv\psi=0\;
\end{align}
Eq.~\eqref{qc} is often referred to as the {\em quantum curve equation} in the literature on topological recursions.
Note that the coefficients of $\log\psi$ enumerate dessins with given numbers of white and black vertices, and a given number of edges regardless of genus (or the number of boundary components).
\end{remark}

Another observation is that the generating function $F=F(s,u,v,p_1,p_2,\dots)$ satisfies an infinite system of non-linear partial differential equations called the KP (Kadomtsev-Petviashvili) hierarchy (this means that the numbers~$N_{k,l}(\mu)$ additionally obey an infinite system of recursions). The KP hierarchy is one of the best studied completely integrable systems in mathematical physics. Below are the first few equations of the hierarchy:
\begin{equation}\label{KP}
\begin{aligned}
&F_{22}=-\frac12\,F_{11}^2+F_{31}-\frac1{12}\,F_{1111}\;,\\
&F_{32}=-F_{11}F_{21}+F_{41}-\frac16F_{2111}\;,\\
&F_{42}=-\frac12\,F_{21}^2-F_{11}F_{31}+F_{51}+\frac18\,F_{111}^2
+\frac1{12}\,F_{11}F_{1111}-\frac14\,F_{3111}+\frac1{120}\,F_{111111}\;,\\
&F_{33}=\frac13\,F_{11}^3-F_{21}^2-F_{11}F_{31}+F_{51}
+\frac14\,F_{111}^2+\frac13\,F_{11}F_{1111}-\frac13\,F_{3111}\\
&\hspace{3.5in}+\frac1{45}\,F_{111111}\;,
\end{aligned}
\end{equation}
where the subscript~$i$ stands for the partial derivative with respect to~$p_i$.

The exponential~$Z=e^F$ of any solution is called a {\em tau function} of the hierarchy. The space of solutions (or the space of tau functions) has a nice geometric interpretation as an infinite-dimensional Grassmannian (called the {\em Sato Grassmannian}), see~\cite{MJD} or~\cite{K} for details. In particular, the space of solutions is homogeneous: there is a Lie algebra $\widehat{\mathfrak{gl}(\infty)}$ that acts infinitesimally on the space of solutions, and the action of the corresponding Lie group is transitive.

\begin{corollary}\label{tau}
The generating function $F=F(s,u,v,p_1,p_2,\dots)$ satisfies the infinite system of KP equations~\eqref{KP} with respect to $p_1,p_2,\dots$ for any parameters~$s,u,v$. Equivalently, the partition function~$Z=e^F$ is a 3-parameter family of KP tau functions.
\end{corollary}
\begin{proof}
To begin with, we notice that~$1$ is obviously a KP tau function. Then, since 
$p_1, \Lambda_1, M_1\in\widehat{\mathfrak{gl}(\infty)}$
(cf.~\cite{K}), the linear combination $s((u+v)\Lambda_1+M_1 +uvp_1)$ also belongs to $\widehat{\mathfrak{gl}(\infty)}$ for any~$s,u,v$. The exponential $e^{s((u+v)\Lambda_1+M_1+uvp_1)}$ therefore preserves the Sato Grassmannian and maps KP tau functions to KP tau functions. Thus,~$e^F$ is a KP tau function as well, and~$F$ is a solution to KP hierarchy.
\end{proof}

\begin{remark}
Corollary~\ref{tau} was earlier proven in~\cite{GJ2} by a different method.
However,~\cite{GJ2} contains no analogs of the Virasoro constraints or the evolution equation.
\end{remark}

At the end of this subsection we will sketch how to enumerate dessins with $k$ white vertices, $l$ black vertices, $d$ edges and $m$ boundary components regardless of the partition $\mu=(\mu_1,\ldots,\mu_m)$. 
To these ends, consider the specialization operator
\begin{align}
\theta: F\mapsto F|_{p_i=t, i=1,2,\ldots}\label{sp}
\end{align}
and put $f=\theta(F)$. This specialization is more subtle than the one considered in Remark \ref{ps}, and the coefficients of $f$ do not mix dessins of different genera since by Euler's formula $(k+l)-d+m=2-2g$. Expanding $f$ into a series in the variables $s$ and $t$, we recompose it as
\begin{align}
f(s,t,u,v)=\sum_{g=0}^\infty\sum_{d=2g+1}^\infty \frac{f_{g,d}(t,u,v)}{d}\,s^d\;,
\end{align}
where each coefficient $f_{g,d}(t,u,v)$ is a homogeneous polynomial in $t,u,v$ of degree $d+2-2g$ with integer coefficients. 

Furthermore, using the Virasoro constraints (\ref{V}) with $n=0,1,2$ and the homgeneity equation
$$s\frac{\partial F}{\partial s}=\sum_{i=1}^\infty ip_i\frac{\partial F}{\partial p_i}\;,$$
we can express the specializations of partial derivatives of $F$ with respect to the variables $p_1,p_2,p_3$ in terms of $s$-derivatives of $f$. More precisely, a straightforward computation yields
\begin{lemma}
We have
\begin{align*}
&\theta(F_{1})=s^2f'+suv, \\
&\theta(F_{11})=s^2(s^2f'+suv)', \\
&\theta(F_{1111})=s^2(s^2(s^2(s^2f'+suv)')')',\\
&2\theta(F_{2})=(1+s(u+v-t))(s^2f'+suv)-suv, \\
&3\theta(F_{3})=2s(u+v-t)\theta(F_{2})+(1-st)\theta(F_{1})+s\theta(F_{11})+s\theta(F_{1})^2-suv, \\
&\theta(F_{12})=s^2\theta(F_{2})', \\
&\theta(F_{13})=s^2\theta(F_{3})', \\
&2\theta(F_{22})=(1+s(u+v-t))\theta(F_{12})+3s\theta(F_{3})-2s\theta(F_{2}),
\end{align*}
where the subscript $i$ stands for the partial derivative with respect to $p_i$, and the prime $'$ denotes the derivative in $s$.
\end{lemma}
Applying the specialization operator $\theta$ to the first KP equation 
$$\theta(F_{22})=-\frac12\,\theta(F_{11})^2+\theta(F_{31})-\frac1{12}\,\theta(F_{1111})\,,$$
cf. (\ref{KP}), and using the above formulas, we get an ordinary differential equation for $f$ as a function of $s$ that translates into the following quadratic recursion:
\begin{align}
(d+1)f_{g,d}&=(2d-1)\,a\,f_{g,d-1}+(d-2)\,b\,f_{g,d-2}+(d-1)^2(d-2)f_{g-1,d-2}\nonumber\\
            &+\sum_{i=0}^g\sum_{j=1}^{d-3}(4+6j)(d-2-j)f_{i,j}f_{g-i,d-2-j}\;,\label{ad}
\end{align}
where $a=t+u+v,\; b=4(tu+tv+uv)-a^2.$ Starting with $f_{0,1}=tuv$, one can recursively compute the polynomials $f_{g,d}$ for all $g$ and $d$. 

Finally, let us restrict ourselves to the case of dessins with one boundary component (or bicolored polygon gluings). Denote by $h_{g,d}$ the linear term in $f_{g,d}$ with respect to $t$. Then the recursion (\ref{ad}) takes the form
\begin{align*}
(d+1)h_{g,d}&=(2d-1)(u+v)h_{g,d-1}-(d-2)(u-v)^2 h_{g,d-2}\\
            &+(d-1)^2(d-2)h_{g-1,d-2}\;,
\end{align*}
and we reproduce the well-known result of \cite{A} on the enumeration of genus $g$ gluings of a bicolored $2d$-gon with given numbers of white and black vertices (cf. also \cite{J}).

\subsection{Virasoro constraints for the numbers of ribbon graphs}
 
A closely related, but somewhat different enumerative problem was considered in~\cite{WL}. Recall that a dessin d'enfant $f^{-1}([0,1])$ is a bicolored connected ribbon graph with vertices ``colored" by either 0 or 1 depending on whether~$f$ maps the vertex to 0 or 1 in~${\C}P^1$. One can similarly try to enumerate all (not necessarily bicolored) connected ribbon graphs, and this is the problem that was addressed in~\cite{WL}. To make it precise, let us label the boundary components of a ribbon graph~$\G$ (or, equivalently, the vertices of the dual graph~$\G^*$) by integers from 1 to~$m$, and let $\mu_1,\ldots,\mu_m$ be the lengths of the boundary components of~$\G$ (or the degrees of vertices of~$\G^*$).

Ribbon graphs can naturally be represented by dessins of a special type called {\em clean dessins} in \cite{DMSS}. Namely, color each vertex of a ribbon graph in white and place black vertices at the midpoints of edges. Such a dessin corresponds to a covering of $\pl$ of even degree $d$ with ramification of type $[2^{d/2}]$ over 1 and arbitrary ramification over 0 and $\infty$.
As before, we put $k=|f^{-1}(0)|$ (the number of vertices of the ribbon graph~$\G$), $l=|f^{-1}(1)|=d/2$ (the number of edges of $\G$), and $m=|f^{-1}(\infty)|$ (the number of boundary components of~$\G$). Clearly, we have $k-d/2+m=2-2g$. 

Denote by $D_{g,m}(\mu)=D_{g,m}(\mu_1,\ldots,\mu_m)$ the number of genus~$g$ ribbon graphs with~$m$ labeled vertices of degrees $\mu_1,\ldots,\mu_m$ counted with weights $\frac{1}{|{\rm Aut}_v\,\G|}$, where the automorphisms preserve each vertex of~$\G$ pointwise. Apparently, the same numbers enumerate pure dessins with $m$ labeled boundary components of lengths $(2\mu_1,\ldots,2\mu_m)$. The following recursion for $D_{g,m}(\mu)$ was derived in~\cite{WL}, Eq.~(6):\footnote{\label{M} This is (a specialization of) Tutte's recursion for arbitrary $g$, cf.~\cite{EO2}. This formula, undeservedly forgotten, was recently reproduced in \cite{DMSS}, Eq.~(3.15). Note that the second term in the r.h.s. of~\eqref{wl}, corresponding to a loop bounding a 1-gon, was inadvertently omitted there. This required some ``modification'' of the numbers $D_{g,m}(\mu_1,\ldots,\mu_m)$ in \cite{DMSS}.}
\begin{align}
\mu_1 D_{g,m}(\mu_1,\ldots,&\mu_m)
=\sum_{j=2}^m (\mu_1+\mu_j-2)D_{g,m-1}(\mu_1+\mu_j-2,\mu_2,\ldots,\widehat{\mu}_j,\ldots,\mu_m)\nonumber\\
&{}+2(\mu_1-2)D_{g,m}(\mu_1-2,\mu_2,\ldots,\mu_m)\nonumber\\
&{}+\sum_{i+j=\mu_1-2}ij \bigg(D_{g-1,m+1}(i,j,\mu_2,\ldots,\mu_m)\nonumber\\
&{}+\sum_{g_1+g_2=g}\quad\sum_{I\sqcup J=\{2,\ldots,m\}}
D_{g_1,|I|+1}(i,\mu_I)D_{g_2,|J|+1}(j,\mu_J)\bigg)\;.\label{wl}
\end{align}
Recursion~\eqref{wl} is valid for all $g\geq 0,\;m\geq 1$, and~$\mu$ such that $d=\sum_{i=1}^m\mu_i>2$, whereas for~$d=2$ the only nonzero numbers are $D_{0,1}(2)=1/2,\;D_{0,2}(1,1)=1$.
Below we give a convenient interpretation of this recursion in terms of PDEs.

Similar to (\ref{gf}), introduce the generating function
\begin{align}\label{gfM}
\tF(s,u,p_1,p_2,\ldots)
&=\sum_{g=0}^\infty\sum_{m=1}^\infty\frac{1}{m!} \sum_{\mu\in{\Z}_+^m} D_{g,m}(\mu)s^d u^k p_{\mu_1}\ldots p_{\mu_m}\;,
\end{align}
where $d=\sum_{i=1}^m\mu_i$, $k=2-2g-m+d/2$, and $\mu=(\mu_1,\ldots,\mu_m)$ (compared to (\ref{gf}), we omit here the trivial factor $v^{d/2}$ that carries no additional information in this case).

\begin{theorem}\label{rg} The generating function $\tF$ enjoys the following integrability properties:
\begin{enumerate}[(i)]
\item Let
\begin{align*}
\tL_n=-\frac{n+2}{s^2}\frac{\partial}{\partial p_{n+2}}
+2\,u\,n\frac{\partial}{\partial p_{n}}
+\sum_{j=1}^\infty p_j(n+j)\frac{\partial}{\partial p_{n+j}}\\
{}+\sum_{i+j=n}ij\frac{\partial^2}{\partial p_i \partial p_j}+\delta_{-1,n}u p_1+\delta_{0,n}u^2\;,
\end{align*}
where $n\geq -1$. Then the partition function $e^\tF$ satisfies the infinite system of PDE's (``Virasoro constraints'')
\begin{align*}
\tL_n e^\tF=0
\end{align*}
that determine $\tF$ uniquely.

\item Put
\begin{align*}
&\Lambda_2=\sum_{i=3}^\infty (i-2)p_i\,\frac{\partial}{\partial p_{i-2}}+\frac12p_1^2\;,\nonumber\\
&M_2=\sum_{i=2}^\infty \sum_{j=1}^{i-1} \left((i-2)p_j p_{i-j}\,\frac{\partial}{\partial p_{i-2}}
+ j(i-j) p_{i+2}\,\frac{\partial^2}{\partial p_j \partial p_{i-j}}\right)\;.
\end{align*}
Then $e^\tF$ satisfies the evolution equation
\begin{align*}
\frac{1}{s}\frac{\partial e^\tF}{\partial s}=(2\,u\,\Lambda_2+M_2+u^2p_2)e^\tF\,,
\end{align*}
that, together with the initial condition~$\tF|_{s=0}=0$,
determines~$\tF$ uniquely. In other words,~$e^\tF$ is explicitly given by the formula
\begin{align*}
e^\tF=e^{\frac{s^2}{2}(2\,u\,\Lambda_2+M_2+u^2p_2)}\,1\;.
\end{align*}

\item The partition function $e^\tF$ is a tau function of the KP hierarchy, i.e. its logarithm $\tF(s,u,p_1,p_2,\ldots)$ satisfies~\eqref{KP} for any $s$ and $u$.

\end{enumerate}
\end{theorem}

\begin{proof}
Part (i) of the theorem is just a reformulation of the recursions (\ref{wl}) for
$\mu_1=n+2$. Note that the operators~$\tL_n$ obey the commutation relations
$[\tL_m,\tL_n]=(m-n)\tL_{m+n}$
for $m>n\geq -1$.

To prove (ii) we multiply $\tL_n$ by~$p_{n+2}$ and sum over~$n$:
\begin{align}
\sum_{n=-1}^\infty p_{n+2}\tL_n
&=\sum_{n=-1}^\infty p_{n+2}\bigg(-\frac{n+2}{s}\frac{\partial}{\partial p_{n+2}}+2\,u\,n\frac{\partial}{\partial p_{n}}\nonumber\\
&\hspace{0.6in}
{}+\sum_{j=1}^\infty p_j(n+j)\frac{\partial}{\partial p_{n+j}}
+\sum_{i+j=n}ij\frac{\partial^2}{\partial p_{i} \partial p_{j}}\bigg)+u\,p_1^2+u^2p_2\nonumber\\
&=-\frac{1}{s}\frac{\partial}{\partial s}+2\,u\,\Lambda_2+M_2+u^2 p_2\;.\label{evom}
\end{align}

Part (iii) follows from the fact that $\Lambda_2, M_2$ and $p_2$ belong to $\widehat{\mathfrak{gl}(\infty)}$,
cf. Corollary~\ref{tau}.
\end{proof}

\begin{remark}
Similar to Eq.~\eqref{qc} we can write the quantum curve equation for the wave function
$$
\widetilde{\psi}=\widetilde{\psi}(s,t,u)=e^{\tF(s,u,p_1,p_2,\ldots)}\left|_{p_i=t^i}\right.\;,
$$
that reads in this case as follows:
$$
t^2\,\frac{d^2\widetilde{\psi}}{dt^2}+\left(2(u+1)\,t-\frac{1}{s^2\,t}\right)\,\frac{d\widetilde{\psi}}{dt}+(u+u^2)\,\widetilde{\psi}=0\;.
$$
Note that this equation differs from the one obtained in \cite{MS}. The reason for that is explained in Footnote \ref{M} above. To fix that, put $Z(x,\hbar)=e^{-\frac{\log x}{\hbar}}\,\widetilde{\psi}|_{s=\hbar/x,t=1,u=1/\hbar}$. Then one has
$$\left(\hbar^2\frac{d^2}{dx^2}+\hbar\frac{d}{dx}+1\right)Z(x,\hbar)=0$$
precisely like in \cite{MS}.
\end{remark}

To complete this section, we will show that the Walsh-Lehman formula (\ref{wl}) implies the Harer-Zagier \cite{HZ} recursion for the numbers of orientable polygon gluings. We will follow the same lines as at the end of the previous subsection. To begin with, put $\tf=\theta(\tF)=\tF|_{p_i=t, i=1,2,\ldots}$, where the specialization operator $\theta$ is defined by Eq.~(\ref{sp}). 
The coefficients of $\tf$ enumerate ribbon graphs with given numbers of vertices, edges and boundary components and, therefore, do not mix graphs of different genera.
Rearrange the series $f$ as follows:
\begin{align}
\tf(s,t,u)=\sum_{g=0}^\infty\sum_{l=2g}^\infty \frac{\tf_{g,l}(t,u)}{2l}\,s^{2l}\;,
\end{align}
where $l=d/2$, and each coefficient $\tf_{g,l}(t,u)$ is a homogeneous polynomial in $t,u$ of degree $l+2-2g$ with integer coefficients. 
Like in the case of dessins, using the Virasoro constraints of Theorem \ref{rg} (i) with $n=-1,0,1$, we can express the specializations of partial derivatives of $\tF$ with respect to the variables $p_1,p_2,p_3$ in terms of $s$-derivatives of $\tf$. A straightforward computation yields
\begin{lemma}
We have
\begin{align*}
&\theta(\tF_{1})=s^3\tf'+s^2tu, \\
&\theta(\tF_{11})=s^3\theta(\tF_{1})'-s^2\theta(\tF_{1})+s^2u, \\
&\theta(\tF_{111})=s^3\theta(\tF_{11})'-2s^2\theta(\tF_{11}),\\
&\theta(\tF_{1111})=s^3\theta(\tF_{111})'-3s^2\theta(\tF_{111}),\\
&2\theta(\tF_{2})=s^3\tf'+s^2u^2, \\
&4\theta(\tF_{22})=s^6\tf''+3s^5\tf'+2s^4u^2, \\
&\theta(\tF_{13})=s^5((2s^2u+1-s^2t)(s\tf')'+(2s^2u+2-s^2t)\tf')+s^6(2tu^2-t^2u)+3s^4u^2, 
\end{align*}
where, as before, the subscript $i$ stands for the partial derivative with respect to $p_i$, and the prime $'$ denotes the derivative in $s$.
\end{lemma}
Applying the specialization operator $\theta$ to the first KP equation (\ref{KP}) 
and using the above formulas, we get an ordinary differential equation for $\tf$ as a function of $s$ that translates into the following quadratic recursion:
\begin{align}
(l+1)\tf_{g,l}&=2(2l-1)(t+u)\tf_{g,l-1}+(2l-1)(2l-3)(l-1)\tf_{g-1,l-2}\nonumber\\
              &+3\sum_{i=0}^g\sum_{j=0}^{l-2}(2j+1)(2(l-2-j)+1)\tf_{i,j}\tf_{g-i,l-2-j}\;,\label{hz}
\end{align}	
where we put by definition $\tf_{0,0}=u$.	This is essentially the formula from \cite{CC}, but derived in a more straightforward way. Note that starting with $\tf_{0,1}=t^2u+tu^2$, one can recursively compute the polynomials $\tf_{g,l}$ for all $g$ and $l$. 

To enumerate genus $g$ ribbon graphs with one boundary component (or $2l$-gon gluings), it is sufficient to consider the linear terms $\epsilon_{g,l}$ in $\tf_{g,l}$ with respect to $t$. Then Eq.~(\ref{hz}) turns into the famous Harer-Zagier recursion
\begin{align*}
(l+1)\epsilon_{g,l}=2(2l-1)u\epsilon_{g,l-1}+(2l-1)(2l-3)(l-1)\epsilon_{g-1,l-2}\;,
\end{align*}
cf. \cite{HZ}.

\section{Topological recursion}

The generating function $F=\sum_{g,m}F_{g,m}$ of \eqref{gf} enumerating Belyi pairs~$(C,f)$ (or Grothendieck's dessins) can be naturally decomposed into components with fixed~$g$ and $m$, where $g$ is the genus of~$C$ and~$m$ is the number of poles of~$f$:
\begin{align}
F_{g,m}(s,u,v,p_1,p_2,\dots)
= \frac{1}{m!}\sum_{\mu_1,\dots,\mu_m}\sum_{k+l=d-m+2-2g}
   N_{k,l}(\mu) s^{d} u^k v^l\, p_{\mu_1}\ldots p_{\mu_m}\label{gm}
	\end{align}
(here, as usual, $d=\sum\mu_i$). Another way to collect these numbers into a generating series is to use the
\emph{$m$-point correlation functions}
$$W_{g,m}(x_1,\dots,x_m)
 =\sum_{\mu_1,\dots,\mu_m}\sum_{k+l=d-m+2-2g} N_{k,l}(\mu)\,x_1^{\mu_1}\dots x_m^{\mu_m}.$$

\emph{Topological recursion} (cf. \cite{EO2}) is an ``ansatz'' that allows to reconstruct the coefficients of certain generating series recursively in~$g$ and~$m$. Traditionally, it is formulated in terms of correlation functions or, rather, 
\emph{differentials}

\begin{align*}
w_{g,m}(x_1,\dots,x_m)&=\frac{\D^mW_{g,m}}{\D x_1\ldots \D x_m} \,dx_1\ldots dx_m\\
&=\sum_{\mu}\sum_{\substack{k,\,l\\k+l=d-m+2-2g}} N_{k,l}(\mu)\prod_{i=1}^m \mu_i x_i^{\mu_i-1} dx_i.
\end{align*}

We will present the topological recursion in terms of components~$F_{g,m}$ of the generating function $F$. The advantage of this approach is that we need only one set of variables~$p_i$ for all~$g$ and~$n$. The two approaches being equivalent, it proves out, however, that many properties of the recursion become more clear in terms of~$p$-variables. 

One of the nice features of topological recursion is that the generating functions~$F_{g,m}$ become polynomilas under a linear change of variables $p_i$. The components~$F_{g,m}$ of the total generating function $F$ are infinite formal series in $p_i$, and their polynomiality is far from being an immediate consequence of the Virasoro constraints for $F$. On the other hand, this polynomiality automatically follows from the equations of topological recursion. Another advantage of topological recursion is its universality. For a variety of enumerative problems it takes the same form, differing only in initial conditions. 

Introduce the formal variables~$x$ and~$z$ related by
\begin{equation}
z(x)=\sqrt{\frac{1-\beta  s  x}{1-\alpha  s  x}},\qquad
  x(z)=\frac{z^2-1}{s  \left(\alpha  z^2-\beta \right)},
\label{eqzx}\end{equation}
where
\begin{equation}
\a=(\sqrt{u}-\sqrt{v})^2,\qquad
 \b=(\sqrt{u}+\sqrt{v})^2.
\label{eqabs}
\end{equation}
We consider~\eqref{eqzx} as a formal change of coordinates on the complex projective line $\pl$ near the point $x=0$ (resp., $z=1$) 
depending on the parameters $\a,\b,s$ (or $u,v,s$).

Put
\begin{equation}
T_j(p)=\sum_{i=1}^\infty c_j^{(i)}p_i\,,
\label{eqT}\end{equation}
where the coefficients $c_j^{(i)}$ are defined by the relation
$$\sum_{i=1}^\infty c_j^{(i)}x^i=(z(x))^{j}-1\,.$$

\begin{theorem}\label{th2}
Let $F_{g,m}$ be the infinite series defined by~\eqref{gm}. Then
\begin{enumerate}[(i)]
\item For each~$g,m$ with~$2g-2+m>0$ there exists a polynomial~$G_{g,m}$ of the variables~$t_j$, $j\in\Z_\odd$, such that
$$F_{g,m}(p)=G_{g,m}(t)\bigm|_{t_j=T_j(p)}$$
(i.e. each~$F_{g,m}$ is a polynomial in the linear functions $T_{\pm1},T_{\pm3},T_{\pm5},\dots$).
\item The polynomials~$G_{g,m}$ can be recursively computed starting from $G_{0,3}$ and $G_{1,1}$, cf. Eqs.~\eqref{eq1}--\eqref{eq4} below.
\end{enumerate}
\end{theorem}

Let us now formulate the recursion for the polynomials~$G_{g,m}$ precisely. This can be done in terms of the so-called \emph{spectral curve}. In our case the spectral curve is the projective line~$\pl$ equipped with the globally defined holomorphic involution $z\mapsto -z$ with respect to some affine coordinate~$z$.

By a {\em Laurent form} we understand here a globally defined meromorphic 1-form on $\pl$ with poles only at~$0$ and~$\infty$. Denote by~$L$ the space of odd Laurent forms relative to the involution $z\mapsto -z$. The forms $d(z^j)=j\,z^{j-1}\,dz$, $j\in\Z_\odd$, provide a convenient basis in~$L$. Let~$P_L$ denote the projector to the space~$L$ in the space of all Laurent forms. For a Laurent form~$\phi$ its projection $\psi=P_L(\phi)$ to $L$ is uniquely determined by the requirement that the form 
$\psi-\phi_\odd$ is regular at both~$0$ and~$\infty$, where $\phi_\odd(z)=\frac12(\phi(z)-\phi(-z))$ is the odd part 
of~$\phi$. More explicitly, the action of~$P_L$ is given by the formula
\begin{align}
(P_L\phi)(z)&=\sum_{i=0}^\infty\res_{w=0}(\phi(w)w^{2i+1})\;z^{-2i-2}dz-
\sum_{i=0}^\infty\res_{w=\infty}(\phi(w)w^{-2i-1})\;z^{2i}dz\nonumber\\
&=\res_{w=0}\left(\phi(w)\frac{w\;dz}{z^2-w^2}\right)+\res_{w=\infty}\left(\phi(w)\frac{w\;dz}{z^2-w^2}\right)\;.
\label{res}\end{align}

Note that for the validity of this definition it will suffice to assume that~$\phi$ is defined in a neighborhood of the 
points~$0$ and~$\infty$, or even that~$\phi$ is  a formal Laurent series at these points. On the other hand, 
the form~$P_L\phi\in L$ is always globally defined on~$\pl$.

In fact, the recursion applies not to the polynomials~$G_{g,m}$ themselves, but to certain $1$-forms~$U_{g,m}$. For the set of variables $z$ and $t=(t_{\pm 1},t_{\pm 3},\ldots)$ introduce the differential operator
$$\d_{z,t}=\sum_{j\in\Z_\odd}jz^{j-1}\pd{}{t_j}\,.$$
For $2g-2+m>0$ put
\begin{align*}
U_{g,m}(z)&=\d_{z,t} G_{g,m}(z)=\sum_{j\in\Z_\odd}jz^{j-1}\pd{G_{g,m}}{t_j}\,.
\end{align*}
As we will see later, $U_{g,m}=U_{g,m}(z)dz$ is an odd Laurent form on~$\pl$ that is polynomial in~$t_j$.

\begin{remark}\label{remdelta}
The operator~$\d_{z,t}$ written in terms of $x$ and $p$-variables becomes
$$\d_{x,p}=\frac{dz}{dx}\sum_{i=1}^\infty ix^{i-1}\,\pd{}{p_i}\,,$$
where $x$ is related to $z$ by (\ref{eqzx}) and $t_j=T_j(p)$, see~\eqref{eqT}. 
(For brevity we will omit the subscripts `$p$' and `$t$' by
$\d$, always associating $p$- and $t$-variables with $x$ and $z$ respectively.)
More precisely, assume that a function~$f$ of $p$-variables can be expressed as a composition  $f=h\circ T$, where $h$ is a function of $t$-variables and~$T$ is the linear change~\eqref{eqT}. Then we have $\d_xf\,dx=(\d_zh)|_{t_k=T_k(p)}\,dz$.
In particular, if~$g$ is a polynomial in $t$-variables, then~$\d_xfdx$ is a Laurent form in~$z$ (with coefficients depending on $p_i$'s). Indeed, the operators on both sides satisfy the Leibnitz rule, and therefore it is sufficient to prove the equality for the case~$h=t_k$, that is,
$$\d_x T_k(p)dx=k z^{k-1} dz,\qquad z=z(x)\,,$$
which is essentially the definition of the linear functions~$T_k(p)$, cf.~(\ref{eqT}).
\end{remark}

In the unstable cases (i.e. when $2g-2+m\leq 0$) the definition of~$U_{g,m}$ should be modified. Namely, we set $U_{0,1}=0$ and 
define~$U_{0,2}$ by the following formal expansions
\begin{equation}
\begin{aligned}
U_{0,2}(z)dz&=-\sum_{i=0}^\infty t_{-2i-1}z^{2i}dz=-(t_{-1}+t_{-3}z^2+\dots)\,dz,\quad z\to 0,\\
U_{0,2}(z)dz&=-\sum_{i=0}^\infty t_{2i+1}z^{-2i}d(z^{-1})=(t_1z^{-2}+t_3z^{-4}+\dots)\,dz,\quad z\to\infty,\\
\end{aligned}
\label{U02}
\end{equation}

In general, the homogeneous degree~$m$ polynomial~$G_{g,m}$ can be recovered form the form~$U_{g,m}$ by the Euler formula
\begin{equation}
G_{g,m}=\frac{1}{m}\sum t_k\pd{G_{g,m}}{t_k}=\frac1m\,\O(U_{g,m},U_{0,2})\;,
\label{eq1}\end{equation}
where for odd forms~$\phi$ and~$\psi$ we set
$$\O(\phi,\psi)=-\O(\psi,\phi)=\res_{z=0}\left(\phi\int\psi\right)+\res_{z=\infty}\left(\phi\int\psi\right)\;.$$

The last ingredient needed to write down the topological recursion is the form
\begin{equation}\label{eta}
\eta=\eta(z)dz=\frac{\s\,z^2dz}{(z^2-1)^2 (\a\,z^2 -\b)},\qquad
\s=(\a-\b)^2=16\,u\,v.
\end{equation}
This form is odd and has the property that the dual vector field
$$\frac{1}{\eta}=\frac{1}{\eta(z)}\,\frac{d}{dz}=\frac{\a\,z^4-(2\,\a+\b)\,z^2+\a +2\,\b -\b\,z^{-2}}{\s}\,\frac{d}{dz}$$
is meromorphic with poles of order~$2$ at~$z=0$ and $z=\infty$ and regular elsewhere in~$\pl$.

The main recursive relation of this paper is
\begin{equation}
U_{g,m}=P_L\Biggl(\frac{1}{2\eta}\Biggl(\d U_{g-1,m+1}+\sum_{\substack{g_1+g_2=g,\\m_1+m_2=m+1}}U_{g_1,m_1}U_{g_2,m_2}\Biggr)\Biggr)
\label{eq2}\end{equation}
(here and below we tacitly assume that $\d f=\d_z f\,dz$).
Note that almost all terms in the sum on the right hand side of (\ref{eq2}) belong to~$L$, so that~$P_L$ is identical on these terms. Therefore, (\ref{eq2}) can equivalently be rewritten as
$$
U_{g,m}=\frac{1}{2\eta}\Biggl(\d U_{g-1,m+1}+\mathop{\sum{}^*}_{\substack{g_1+g_2=g,\\m_1+m_2=m+1}} 
U_{g_1,m_1}U_{g_2,m_2}\Biggr)+P_L\left(\frac1\eta\,U_{g,m-1}U_{0,2}\right)\;,
$$
where the star $^*$ by the summation sign means that the terms involving~$U_{0,2}$ are excluded (recall that $U_{0,1}=0$ by assumption). This recursion relation is valid for all~$g$ and~$n$ with~$2g-2+m>1$. Moreover, it applies for~$(g,m)=(0,3)$ as well:
\begin{equation}
U_{0,3}=P_L\left(\frac{U_{0,2}^2}{2\eta}\right).
\label{eq3}\end{equation}
In the case $(g,m)=(1,1)$ the formula is not applicable since~$\d U_{0,2}$ is not defined. This is why we set by definition
\begin{equation}
U_{1,1}=U_{1,1}(z)\,dz=P_L\left(\frac{1}{2\eta}\left(\frac{dz}{2z}\right)^2\right)=\frac{1}{2\eta}\left(\frac{dz}{2z}\right)^2
=\frac{1}{8z^2\eta(z)}\,dz\,.
\label{eq4}\end{equation}

Eqs.~\eqref{eq1}--\eqref{eq4} concretize the second statement of Theorem~\ref{th2}.

Below we list the polynomials $U_{g,m}$ and $G_{g,m}$ for small~$g$ and~$m$:

\begin{align*}
\s\,U_{0,3}(z)&=\frac{\a}{2}\,t_1^2-\frac{\b}{2}\,t_{-1}^2\,z^{-2},\\[-12pt]\\
\s\,G_{0,3}&=\frac{\a}{6}\,t_1^3+\frac{\b}{6}\,t_{-1}^3,\\[-12pt]\\
\s\,U_{1,1}(z)&=\frac{\a}{8}\,z^2 -\frac{2\a+\b}{8}
  +\frac{\a+2\b}{8}\,z^{-2}-\frac{\b}{8}\,z^{-4},\\[-12pt]\\
\s\,G_{1,1}&=\frac{\a}{24}\,t_3-\frac{2\a+\b}{8}\,t_1-\frac{\a+2\b}{8}\,t_{-1}+\frac{\b}{24}\,t_{-3},\\[-12pt]\\
\s^2\,U_{0,4}(z)&=\frac{\a^2}{2}t_1^3\,z^2
 +\Bigl(\frac{\a^2}{2}\,t_3\, t_1^2-\frac{\a(2\a+\b)}{2}\,t_1^3-\frac{\a\b}{2}\,t_{-1}^2\, t_1\Bigr)\\
 &+\Bigl(\frac{\b(\a+2\b)}{2}\,t_{-1}^3+\frac{\a\b}{2}\,t_1^2\, t_{-1}-\frac{\b^2}{2}\,t_{-3}\, t_{-1}^2\Bigr)\,z^{-2}
 -\frac{\b^2}{2}\,t_{-1}^3\,z^{-4},\\[-12pt]\\
\s^2\,G_{0,4}&=\frac{\a^2}{6}\,t_1^3\, t_3-\frac{\a(2\a+\b)}{8}\,t_1^4-\frac{\a\b}{4}\,t_1^2\, t_{-1}^2
     -\frac{\b(\a+2\b)}{8}\,t_{-1}^4+\frac{\b^2}{6}\,t_{-3}\,t_{-1}^3,\\[-12pt]\\
\s^2\,U_{1,2}(z)&=\frac{5\a^2}{8}\,t_1\,z^4+\Bigl(\frac{\a^2}{8}\,t_3-\frac{3\a(2\a+\b)}{4}\,t_1\Bigr)\,z^2\\
&+\Bigl(\frac{10\a^2+16\a\b+\b^2}{8}\,t_1+\frac{\a^2}{8}\,t_5+\frac{\a\b}{2}\,t_{-1}-\frac{\a(2\a+\b)}{4}\,t_3\Bigr)\\
&+\Bigl(-\frac{\a^2+16\a\b+10\b^21}{8}\,t_{-1}+\frac{\b(\a+2\b)}{4}\,t_{-3}-\frac{\a\b}{2}\,t_1-\frac{\b^2}{8}\,t_{-5}\Bigr)\,z^{-2}\\
&+\Bigl(\frac{3\b(\a+2\b)}{4}\,t_{-1}-\frac{\b^2}{8}\,t_{-3}\Bigr)\,z^{-4} 
  -\frac{5\b^2}{8}\,t_{-1}\,z^{-6},\\[-12pt]\\
\s^2\,G_{1,2}&=\frac{\a^2}{8}\,t_1\,t_5+\frac{\a^2}{48}\,t_3^2-\frac{\a(2\a+\b)}{4}\,t_1\,t_3+\frac{\a^2+16\a\b+10\b^2}{16}\,t_{-1}^2\\
&+\frac{\a\b}{2}\,t_{-1}\,t_1+\frac{10\a^2+16\a\b+\b^2}{16}\,t_1^2-\frac{\b(\a+2\b)}{4}\,t_{-3}\,t_{-1}\\
&+\frac{\b^2}{48}\,t_{-3}^2+\frac{\b^2}{8}\,t_{-5}\,t_{-1}\;,
\end{align*}
where $\a,\b,\s$ are given by~\eqref{eqabs},~\eqref{eta}. This list can be continued further on.

\begin{remark}
Here we compare our form of the topological recursion~\eqref{eq2} with the one that can be found in the literature, see, e.g.~\cite{EO1}, \cite{EO2}.
\begin{enumerate}[(i)]
\item Traditionally, the spectral curve comes with an embedding to (a compactification of)~$\C^2$ by means of certain meromorphic functions~$X,Y$ on~$\pl$. These functions are chosen so that~$X$ is even with respect to the involution $z\mapsto -z$, and~$\eta=Y\,dX$. In our case we could have set, for example,
$$X=\frac1x=s\,\frac{\a z^2-\b}{z^2-1},\qquad Y=-\frac{\a-\b}{2s}\,\frac{z}{\alpha z^2-\beta}.$$
The formulas of the topological recursion, however, involve the coordinates~$X$ and~$Y$ only in the combination~$\eta=Y\,dX$.
\item The topological recursion is usually formulated in terms of $m$-point correlators. They are related to the homogeneous components~$F_{g,m}$ of the generating function~$F$ by the formulas
\begin{align}
w_{g,m}(x_1,\dots,x_m)&=\d_{x_1}\ldots\d_{x_m} F_{g,m}\,dx_1\ldots dx_m\nonumber\\
&=\d_{x_2}\ldots\d_{x_m} U_{g,m}(x_1)\,dx_1\ldots dx_m\;,
\label{eqwgn}
\end{align}
where $\d_{x_j}=\sum_{i=1}^\infty i\,x_j^{i-1}\,\pd{}{p_i}$. Via the change of variables~\eqref{eqzx}, $w_{g,m}$ can be viewed as a meromorphic $m$-differential on $\left(\pl\right)^m$ that is a Laurent form with respect to each of its arguments provided~$2g-2+m>0$.
\item A version of~\eqref{eqwgn} holds also for~$(g,m)=(0,2)$. Namely, $w_{0,2}(z_1,z_2)=\d_{z_2}U_{0,2}\,dz_2$ is the odd part of the \emph{Bergman kernel} $B(z_1,z_2)=\frac{dz_1\,dz_2}{(z_1-z_2)^2}$:
$$w_{0,2}(z_1,z_2)=\d_{z_2}U_{0,2}\,dz_2=\frac12\left(B(z_1,z_2)-B(z_1,-z_2)\right).$$
This can be interpreted as an equality of asymptotic expansions of the left and right hand sides at $z_1\to 0$ and~$z_1\to\infty$,
cf.~\eqref{U02}: 
$$w_{0,2}(z_1,z_2)=
\begin{cases}-\sum_{i=0}^\infty d(z_2^{-2i-1})\;z_1^{2i}dz_1,&z_1\to 0\;,\medskip\\ 
-\sum_{i=0}^\infty d(z_2^{2i+1})\;z_1^{-2i}d(z_1^{-1}),& z_1\to\infty\;.\end{cases}$$
\item The projector~$P_L$, see \eqref{res}, is given by the contour integral
\begin{align*}
(P_L\phi)(z)\,dz=\frac{1}{2\pi\sqrt{-1}}\left(\int_{|w|=\epsilon}\frac{\phi(w)\,w\,dw}{z^2-w^2}
+\int_{|w|=1/\epsilon}\frac{\phi(w)\,w\,dw}{z^2-w^2}\right)\,dz\;,
\end{align*}
for small $\epsilon > 0$, where
$$\frac{w\;dz}{z^2-w^2}=\frac12\int\limits_{-w}^w B(z,\cdot).$$
This explains the appearance of the Bergman kernel in the majority of expositions of the topological recursion.
\item The above items (i)--(iv) demonstrate that the traditional form of the topological recursion in terms of the correlators~$w_{g,n}$ is obtainable from~\eqref{eq2} by applying $\d_{z_2}\ldots\d_{z_n}$ to the both sides of it.
\end{enumerate}
\end{remark}

\section{Proof of the topological recursion}

\subsection{Master Virasoro equation}

As we will see below, the topological recursion relations are just the equivalently reformulated Virasoro constraints.
To begin with, let us collect the Virasoro constraints~\eqref{vc} into a single equation by multiplying the~$n$th equation by~$x^n$ (where~$x$ is a formal variable) and summing them up:
\begin{multline*}
\sum_{n=0}^\infty x^n\Biggl(-\frac1s(n+1)\pd{F}{p_{n+1}}+(u+v)\,n\pd{F}{p_n}+
\sum_{j=1}^\infty p_j(n+j)\pd{F}{p_{n+j}}\Biggr.\\\Biggl.+
\sum_{i+j=n}ij\left(\pd{^2F}{p_i\D p_j}+\pd{F}{p_i}\pd{F}{p_j}\right)\Biggr)+uv=0.
\end{multline*}
This equation can be simplified. Notice that
\begin{align*}
\sum_{n=0}^\infty&\; x^n\left(-\frac1s(n+1)\pd{F}{p_{n+1}}+(u+v)\,n\pd{F}{p_n}\right)
  =\left(-\frac1s+(u+v)\,x\right)\d_x F\;,\\
\sum_{n=0}^\infty&\; x^n\sum_{i+j=n}ij\left(\pd{^2F}{p_i\D p_j}+\pd{F}{p_i}\pd{F}{p_j}\right)=
 x^2\left(\d_x^2 F+(\d_x F)^2\right)\;,
\end{align*}
where, as in the previous section, $\d_x=\sum_{n=1}^\infty nx^{n-1}\pd{}{p_n}$ (cf. Remark \ref{remdelta}).
As for the third term in the sum, we use the identity $p_j=\d_y^{-1}(j\,y^{j-1})=\d_y^{-1}d_y(y^j)$, where~$y$ is a new independent formal variable and $d_y=\frac{d}{dy}$, to re-write it as follows:
\begin{align*}
\sum_{n=0}^\infty\; x^n\sum_{j=1}^\infty& p_j(n+j)\pd{F}{p_{n+j}}\\
   &=\sum_{n=0}^\infty\;\sum_{i+j=n} x^{i}\,p_j\,n\,\pd{F}{p_n}
   =\d_y^{-1} d_y\left(\sum_{n=0}^\infty\;\sum_{i+j=n} x^{i}\,y^j\,n\,\pd{F}{p_n}\right)\\
   &=\d_y^{-1} d_y\left(\sum_{k=0}^\infty \frac{x^{k+1}-y^{k+1}}{x-y}\, k\,\pd{F}{p_k}\right)
   =\d_y^{-1} d_y\left(\frac{x^2{\d_x F}-y^2\d_y F}{x-y}\right)
\end{align*}
This yields the following \emph{master Virasoro equation} that unifies all Eqs.~\eqref{vc}:
\begin{align}
\Bigl(-\frac1s+(u+v)\,x\Bigr)\d_x F&+x^2\left(\d_x^2 F+(\d_x F)^2\right)\nonumber\\
&+\d_y^{-1} d_y\left(\frac{x^2\d_x F-y^2\d_y F}{x-y}\right)+uv=0\;.
\label{mve}
\end{align}

\subsection{Unstable terms and the spectral curve}

Our immediate goal is to extract the homogeneous terms in~\eqref{mve} contributing to $\d_x F_{g,n}$ for fixed~$g$ and~$n$. We start with the unstable cases. For~$g=0$ and~$n=1$ we get
\begin{equation}
x^2\,(\d_x F_{0,1})^2+\Bigl(-\frac1{s}+(u+v)\,x\Bigr)\d_x F_{0,1}+u v=0\,.\label{eqspectr}
\end{equation}
Solving this equation for $x\,\d_x F_{0,1}$, choosing the proper root and expanding it into the Taylor series at~$x=0$ we get
\begin{align*}
x\,\d_x F_{0,1}&=\frac{1}{2} \left(\frac{1}{s x}-u-v-\sqrt{\left(\frac{1}{s x}-u-v\right)^2-4 u v}\right)\\
&=u\,v\,s\,x+u\,v\,(u+v)\,s^2x^2+ u\,v\,\left(u^2+3\,u\,v+v^2\right)\,s^3x^3 +\dots
\end{align*}
and
$$F_{0,1}=u\,v\,s\,p_1+\frac12u\,v\,(u+v)\,s^2p_2+\frac13u\,v\,\left(u^2+3\,u\,v+v^2\right)\,s^3p_3 +\dots.$$

\begin{definition} The \emph{spectral curve} is the Riemann surface of the algebraic function $x\,\d_x F_{0,1}$ in the $x$-variable.
\end{definition}

In other words, the spectral curve is an algebraic curve such that~$x$ and~$x\d_x F_{0,1}$ are globally defined univalued meromorphic functions on it. In our case the spectral curve is given by~\eqref{eqspectr}. It is rational (admits a rational parametrization). Let~$z$ be an affine coordinate on~$\pl$. Its choice is not important, but, for convenience, we choose it in such a way that the two critical points of the function~$x$ on~$\pl$ are~$z=0$ (with the critical value~$x=\frac{1}{s\b}$) and~$z=\infty$ (with the critical value~$x=\infty$). The corresponding rational parametrization has the following form:
\begin{align*}
x&=\frac{z^2-1}{s\,(\a  z^2-\b)},&
z&=\sqrt{\frac{1-\b  s x}{1-\a s x}},\\
x\,\d_x F_{0,1}&=\sqrt{uv}\,\frac{1-z}{1+z},&\d_x F_{0,1}\frac{dx}{dz}&=\frac{8\,u\,v\,z}{(z+1)^2 \left(\alpha z^2-\beta \right)},
\end{align*}
where~$\a,\b$ are related to~$u,v$ by~\eqref{eqabs}.
All functions entering these equalities can be regarded either as rational functions in $z$-variable or as formal power expansions of these functions at~$z=1$.

We continue with the terms with~$g=0$ and~$n=2$ in~\eqref{mve}. We have
$$\Bigl(-\frac1s+(u+v)x\Bigr)\d_x F_{0,2}+2\,x^2\,\d_x F_{0,1}\,\d_x F_{0,2}
+\d_y^{-1} d_y\left(\frac{x^2\,\d_x F_{0,1}-y^2\,\d_y F_{0,1}}{x-y}\right)
=0\,,$$
from where we get
$$\d_y\,\d_x\,F_{0,2}=\frac{d_y\left(\frac{x^2\,\d_x\, F_{0,1}-y^2\,\d_y\, F_{0,1}}{x-y}\right)}
     {\frac1s-(u+v)x-2\,x^2\,\d_x F_{0,1}}\;.$$
This equality uniquely determines~$\d_y\d_x F_{0,2}dxdy$ as a meromorphic bidifferential on~$\pl\times\pl$. Substituting the obtained above expressions for~$x,\;dx/dz$ and $\d_x F_{0,1}$ into the last formula, after some miraculous cancellations we finally get
\begin{equation}
\d_y\d_x F_{0,2}\,dx\,dy=\frac{dz\,dw}{(z+w)^2}=-B(z,-w),
\label{ddF02}\end{equation}
where $B(z,w)=\frac{dz\,dw}{(z-w)^2}$ is the Bergman kernel, and $w$ is related to $y$ by the same formulas \eqref{eqzx} that relate $z$ to $x$.

\subsection{Rational recursion formula}

Now we look at the homogeneous terms of genus $g$ and degree $m$ with $2g-2+m>1$ in \eqref{mve}. To begin with, let us extract the unstable terms from the expression~$(\d F)^2$ in~\eqref{mve} in order to re-group them with the other summands. Multiplying~\eqref{mve} by $x^{-2}$ we get
\begin{align*}
&2\left(-\frac{1}{2\,s\,x^2}+\frac{u+v}{2\,x}+\d_x F_{0,1}\right)\,\d_x(F-F_{0,1}-F_{0,2})+\d_x^2 F\\
&\quad +\bigl(\d_x(F-F_{0,1}-F_{0,2})\bigr)^2+x^{-2}\d_y^{-1}\,d_y\,\left(\frac{x^2\,\d_x(F-F_{0,1})-y^2\,\d_y(F-F_{0,1})}{x-y}\right)\\
&\hspace{2.4in}+2\d_x(F-F_{0,1})\,\d_x F_{0,2}-(\d_x F_{0,2})^2=0.
\end{align*}
Let us rewrite the coefficients of this equation in the $z$-coordinate. It is convenient to put
$$
\eta=\eta(x)\,dx=-\left(-\frac{1}{2\,s\,x^2}+\frac{u+v}{2\,x}+\d_x F_{0,1}\right)\,dx\;,
$$
so that in terms of the coordinate $z$
$$
\eta=\eta(z)\,dz=\frac{(\a-\b)^2\,z^2}{(z^2-1)^2 (\a\,z^2 -\b)}\,dz\,. 
$$
Then, using the already known expressions for~$x$, $\frac{dx}{dz}$,~$\d_x F_{0,1}$, and~$\d_y\d_x F_{0,2}$, we find that
\begin{align}
\d_x F_{0,2}=-\d_y^{-1}\,d_y\left(\frac{1}{z+w}\right).
\label{delf}
\end{align}
From the above identities we obtain the following equation for~$\d_x F$:
\begin{align}
\d_x(F-F_{0,1}-F_{0,2})&=\frac{1}{2\,\eta(x)}\Bigl(\d_x^2 F+(\d_x(F-F_{0,1}-F_{0,2}))^2-(\d_xF_{0,2})^2\Bigr)\nonumber\\
+\d_y^{-1}\,d_y&\left(\frac{w}{z^2-w^2}\left(\frac{\d_x(F-F_{0,1})}{\eta(x)}-\frac{\d_y (F-F_{0,1})}{\eta(y)}\right)\right)\,
\frac{dz}{dx}\;.\label{eqvirz}
\end{align}
In particular, for the homogeneous components $U_{g,m}=\d_x F_{g,m}$ with~$2g-2+m>1$ this equation reads
\begin{multline}
U_{g,m}(z)=\frac{1}{2\,\eta(z)}\Biggl(\d_z U_{g-1,m+1}(z)+\mathop{\sum{}^{\mathrlap{*}}}_{\substack{g_1+g_2=g,\\m_1+m_2=m+1}}
U_{g_1,m_1}(z)\,U_{g_2,m_2}(z)\Biggr)\\
 {}+\d_w^{-1}\,d_w\left(\frac{w}{z^2-w^2}\left(\frac{U_{g,m-1}(z)}{\eta(z)}-\frac{U_{g,m-1}(w)}{\eta(w)}\right)\right)\;,
 \label{eqrat}\end{multline}
where the star $^*$ by the summation sign means that the unstable terms with $2g-2+m\leq 0$ are omitted. 
We refer to this equation as the \emph{rational recursion formula} for the forms~$U_{g,m}=U_{g,m}\,(z)dz$. 

This equation allows to prove the polynomiality property for the forms~$U_{g,m}$ by induction in~$g$ and~$m$. Indeed, assume that~$U_{g',m'}$ is a Laurent form in~$z$ with coefficients polynomially depending on~$t_k=T_k(p)$, $k\in Z_\odd$, for all~$(g',m')$ with $0<2g'-2+m'<2g-2+m$. Then the first summand in~\eqref{eqrat} obviously also has this form. Let us examine the second summand. Note that the operator $\d_w^{-1}\,d_w$ is well defined on the space of odd Laurent polynomials in variable $w$, so let us check that this condition always holds. Indeed,
the function~$\frac{U_{g,m-1}(z)}{\eta(z)}$ is an even Laurent polynomial in~$z$, therefore, it can be represented as a Laurent polynomial in~$z^2$. Therefore, $\frac1{z^2-w^2}\left(\frac{U_{g,m-1}(z)}{\eta(z)}-\frac{U_{g,m-1}(w)}{\eta(w)}\right)$ is a Laurent polynomial in~$z^2$ and~$w^2$ regular at~$z=\pm w$.  Multiplying it by~$w$ and applying~$d_w$ we obtain an odd Laurent form in~$w$. The polynomiality property for the forms~$U_{g,m}$ follows now from Remark~\ref{remdelta}.

Utilising Remark~\ref{remdelta} once again, we obtain the polynomiality property for~$F_{g,m}$ with~$2g-2+m>0$ as well. This proves the main assertion of Theorem~\ref{th2} (under the assumption that it is valid for the initial terms~$F_{0,3}$ 
and~$F_{1,1}$; this is checked below).

\subsection{The residual formalism}

Since the both sides of~\eqref{eqrat} belong to~$L$, it is convenient to equate the projections of the terms entering this equality by applying~$P_L$ to each of them. The terms of the first summand on the right hand side already belong to~$L$, so~$P_L$ is identical on them. Compute the image of the second summand on the right hand side under the projection~$P_L$. The key observation is that~$P_L$ can be applied to the two terms of this summand separately. In particular, the form~$\frac{dz}{z^2-w^2}$ is regular both at~$z=0$ and at~$z=\infty$ so that it does not contribute to the image. It remains to compute the term
\begin{align*}
&P_L\!\left(\d_w^{-1}d_w\!\left(\frac{w}{z^2-w^2}\frac{U_{g,m-1}(z)}{\eta(z)}\right)dz\right)\\
&\hspace{2in}=P_L\!\left(\frac{U_{g,m-1}(z)}{\eta(z)}\,\d_w^{-1}\!\left(\frac{(z^2+w^2)}{(z^2-w^2)^2}dz\,dw\right)\right).
\end{align*}
The form $\frac{(z^2+w^2)}{(z^2-w^2)^2}\,dz\,dw$ is not Laurent so that~$\d_w^{-1}$ is not applicable to it directly. However, what we actually need in order to apply~$P_L$ is the expansion of this form at~$z=0$ and~$z=\infty$. The coefficients of these expansions are Laurent with respect to~$w$:
\begin{align*}
&\frac{(z^2+w^2)}{(z^2-w^2)^2}dz\,dw=-\sum_{i=0}^\infty d(w^{-2i-1})z^{2i}dz=
  -\d_w\sum_{i=0}^\infty t_{-2i-1}z^{2i}dz,\quad z\to 0,\\
&\frac{(z^2+w^2)}{(z^2-w^2)^2}dz\,dw=-\sum_{i=0}^\infty d(w^{2i+1})z^{-2i}d(z^{-1})\\
&\hspace{2.2in}=-\d_w\sum_{i=0}^\infty t_{2i+1}z^{-2i}d(z^{-1}),\quad z\to \infty.
\end{align*}
In other words, the form $\d_w^{-1}\Bigl(\frac{(z^2+w^2)}{(z^2-w^2)^2}\,dz\,dw\Bigr)$ is well defined in some neighborhoods of the points~$z=0$ and~$z=\infty$ and coincides with the form~$U_{0,2}$ defined by~\eqref{U02}. This proves the equality~\eqref{eq2} of the topological recursion.

\subsection{Initial terms of the recursion}

In order to finish the proof of Theorem~\ref{th2}, it remains to check it for the initial terms of the recursion, that is, for~$U_{1,1}=\d_x F_{1,1}\,dx$ and~$U_{0,3}=\d_x F_{0,3}dx$.

For the case~$g=m=1$, equating the corresponding terms in~\eqref{eqvirz} we get
$$U_{1,1}=\frac{1}{2\eta(x)}\d_x^2F_{0,2}\,dx\,,$$
and using the explicit formula~\eqref{ddF02} for~$\d^2F_{0,2}$, we obtain the required formula~\eqref{eq4} for~$U_{1,1}$.

For the case~$g=0$,~$m=3$ the computations are slightly more involved. Eq.~\eqref{eqvirz} uniquely determines the form~$U_{0,3}$ as
$$U_{0,3}=
 \d_w^{-1}\, d_w\left(\frac{w}{z^2-w^2}\left(\frac{\d_x F_{0,2}(x)}{\eta(x)}-\frac{\d_y F_{0,2}(y)}{\eta(y)}\right)\right)\,dz
 -\frac{(\d_x F_{0,2})^2}{2\,\eta(x)}\,dx\;,
$$
where, as before, $y$ is related to $w$ by the same formulas \eqref{eqzx} that relate $x$ to $z$.

The form~$\d_x F_{0,2}\,dx$ being known, cf.~\eqref{delf}, we directly compute
$$U_{0,3}=\frac{1}{32 u v}\left(\a\,t_1^2dz-\b t_{-1}^2\frac{dz}{z^2} \right)=P_L\left(\frac{(U_{0,2}(z))^2}{2\eta(z)}\,dz\right)$$
which agrees with~\eqref{eq3}. This completes the proof of Theorem~\ref{th2}.

\section{Topological recursion for ribbon graphs}

Here we discuss the toplogical recursion for the numbers $D_{g,m}$ of genus $g$ ribbon graphs with $m$ labeled boundary components of lengths $\mu_1,\ldots,\mu_m$, see Section 2. A topological recursion for these numbers was first obtained in~\cite{EO1}, 
Theorem 7.3 (it was later rediscovered in \cite{DMSS}). We present a simple proof of this recursion in terms of the homogeneous components~$\tF_{g,m}$ of the generating function~\eqref{gfM}
that follows directly from the Virasoro constraints, cf. Theorem~\ref{rg},~(i). In fact, the argument is quite parallel to that of the case of dessins d'enfants. This is why we skip the details paying attention only to the differences between these two enumerative problems. More specifically, we have the same spectral curve $\C P^1$ equipped with an affine coordinate~$z$ and the 
involution~$z\mapsto-z$. What is different, is the choice of the local coordinate~$x$ at the point~$z=1$ and the form~$\eta$.

Consider the linear functions $\tT_k(p)$ given by~\eqref{eqT} with
\begin{equation}\label{zxM}
z(x)=\sqrt{\frac{1+2\sqrt{u}\; s\, x}{1-2\sqrt{u}\; s\, x}}
\end{equation}

\begin{theorem}\label{th3}
Let $\tF_{g,m}$ be the infinite series defined by
\begin{align*}
\tF_{g,m}(s,u,p_1,p_2,\ldots)
&=\frac{1}{m!} \sum_{\mu\in{\Z}_+^m} D_{g,m}(\mu)s^d u^k p_{\mu_1}\ldots p_{\mu_m}\;.
\end{align*} 
Then
\begin{enumerate}[(i)]
\item For each~$g,m$ with~$2g-2+m>0$ there exist a polynomial~$\tG_{g,m}$ of the variables~$t_j$, $j\in\Z_\odd$, such that
$$\tF_{g,m}(p)=\tG_{g,m}(t)\bigm|_{t_j=\tT_j(p)}$$
(i.e. each~$\tF_{g,m}$ is a polynomial in the linear functions $\tT_{\pm1},\tT_{\pm3},\tT_{\pm5},\ldots$).
\item The polynomials~$\tG_{g,m}$ can be recursively computed, starting from $\tG_{0,3}$ and $\tG_{1,1}$, by the same Eqs.~\eqref{eq2}--\eqref{eq4} with $\eta$ given by the formula
$$
\eta=\eta(z)\,dz=-\frac{16uz^2dz}{(1-z^2)^3}\;.
$$
\end{enumerate}
\end{theorem}

\begin{proof}
Like in the case of dessins, we start with the master Virasoro equation that readily follows from Theorem~\ref{th2},~(i):
\begin{align*}
\Bigl(-\frac1{s^2}+2\,u\,x^2\Bigr)\,\d_x \tF&+x^3\Bigl(\d_x^2 \tF+(\d_x \tF)^2\Bigr)\\
&+\d_y^{-1}\,d_y\left(\frac{x^3\,\d_x \tF-y^3\,d_y \tF}{x-y}\right)+u^2\,x+u\,p_1=0\;.
\end{align*}
The spectral curve in this case (an analog of \eqref{eqspectr} above) is given by the equation
$$x^2\,(\d_x\tF_{0,1})^2-x\,\d_x\tF_{0,1}\Bigl(\frac{1}{s^2 x^2}-2 u\Bigr)+u^2=0\,.$$
Solving this equation for $x\,\d_x\tF_{0,1}$ we get the following rational parametrization of the spectral curve
\begin{align*}
&x(z)=\frac{z^2-1}{2s\sqrt{u}\;(z^2+1)}\;,\\
&x\,\d_x\tF_{0,1}=\frac{1-2\,s^2\,u\,x^2-\sqrt{1-4\,s^2\,u\,x^2}}{2s^2\,x^2}=u\left(\frac{1-z}{1+z}\right)^2\;,
\end{align*}
with the inverse change $z(x)$ given by~\eqref{zxM}.
The genus $0$ two point correlator is the same as in the case of dessins:
$$\d_y\d_x\tF_{0,2}\,dx\,dy=\frac{dz\,dw}{(z+w)^2},$$
and instead of the form $\eta$ we have
$$\tilde\eta=\tilde\eta(x)dx=\left(-\d_x \tF_{0,1}-\frac{u}{x}+\frac{1}{2\,s^2\,x^3}\right)\,dx=-\frac{16\,u\,z^2}{(1-z^2)^3}\,dz\;.$$
Thus, the master Virasoro equation in $z$-coordinate acquires the same form~\eqref{eqvirz} and implies the topological recursion~\eqref{eqrat} with $\eta$ replaced by $\tilde\eta$.
\end{proof}

Below are the first few terms of the recursion:
\begin{align*}
U_{0,3}&=\frac{1}{32\,u}\left(t_1^2-t_{-1}^2z^{-2}\right)\,,\\[-12pt]\\
G_{0,3}&=\frac{1}{96\,u} \left(t_{1}^3+t_{-1}^3\right)\,,\\[-12pt]\\
U_{1,1}&=\frac{1}{128\,u}\left(z^2-3+3\,z^{-2}-z^{-4}\right)\,,\\[-12pt]\\
G_{1,1}&=\frac{1}{384\,u}\left(t_{3}-9 t_{1}-9 t_{-1}+t_{-3}\right)\,,\\[-12pt]\\
U_{0,4}&=\frac{1}{512\,u^2}\left(t_1^3 z^2+(t_3 t_1^2-3 t_1^3-t_{-1}^2 t_1)\right.\\[-12pt]\\
&\hspace{.9in}\left.+(t_1^2 t_{-1}+3 t_{-1}^3-t_{-3} t_{-1}^2)\,z^{-2}-t_{-1}^3\,z^{-4}\right)\,,\\[-12pt]\\
G_{0,4}&=\frac{1}{6144\,u^2}\left(4 t_1^3 t_3-9 t_1^4-6 t_1^2 t_{-1}^2-9 t_{-1}^4+4 t_{-3} t_{-1}^3\right)\,,\\[-12pt]\\
U_{1,2}&=\frac{1}{2048\,u^2}\left(5\,t_1\,z^4+(t_3-18 t_1)\,z^2+(4 t_{-1}+27 t_1-6 t_3+t_5)\right.\\[-12pt]\\
 &\hspace{.3in}\left.+(-t_{-5}+6 t_{-3}-27 t_{-1}-4 t_1)\,z^{-2}+(18 t_{-1}-t_{-3})\,z^{-4}-5\,t_{-1}\,z^{-6}\right)\,,\\[-12pt]\\
G_{1,2}&=\frac{1}{12288\,u^2}\left(6 t_1 t_5+t_3^2-36 t_1 t_3+81 t_1^2+24 t_{-1} t_1\right.\\[-12pt]\\
 &\hspace{2in}\left.+81 t_{-1}^2-36 t_{-1} t_{-3}+t_{-3}^2+6 t_{-5} t_{-1}\right)\,.
\end{align*}

\medskip
\noindent
{\bf Acknowledgments.}
The main results of the paper, Theorems 4 and 5, were obtained under support of the Russian Science Foundation grant 14-21-00035. The work of MK was additionally supported by the President of Russian Federation grant NSh-5138.2014.1 and by the RFBR grant 13-01-00383. PZ acknowledges hospitality of the Center for Quantum Geometry of Moduli Spaces at Aarhus University. We thank JSC ``Gazprom Neft" for funding short-term visits of MK to the Chebyshev Laboratory at SPbSU. We are grateful to L.~Chekhov, B.~Eynard, P.~Norbury and G.~Schaeffer for useful discussions, and to the anonymous referee for correcting a few typos and suggesting several improvements in the text.

\end{document}